\documentclass{article}
\usepackage{graphicx} 

\usepackage[top=2.54cm, bottom=2.54cm, left=3.17cm, right=3.17cm]{geometry}
\usepackage{amsmath}    
\usepackage{amsfonts}    
\usepackage{amsthm} 
\usepackage{mathtools} 
\usepackage{amssymb} 
\usepackage{enumitem} 
\usepackage{tikz} 
\usepackage{subcaption} 
\usepackage{indentfirst} 
\usepackage{hyperref} 

\usepackage[dvipsnames]{xcolor}

\tikzset{every picture/.style={line width=0.75pt}}

\newtheorem{theorem}{Theorem}[section]
\newtheorem{lemma}[theorem]{Lemma}
\newtheorem{observation}[theorem]{Observation}
\newtheorem{claim}[theorem]{Claim}

\theoremstyle{remark}

\newtheorem*{remark*}{Remark}

\DeclareMathOperator*{\conv}{conv }

\title{A note on piercing discrete rectangles}
\author{
    Wei Rao\thanks{Moscow Institute of Physics and Technology, Institutsky lane 9, Dolgoprudny, Moscow region, 141700, Russia. Email address: \href{mailto:raowei1998@gmail.com}{raowei1998@gmail.com}} 
}

\date{}

\begin{document}

\maketitle

\begin{abstract}
In 2008, Halman proved a discrete Helly-type theorem for axis-parallel boxes in $\mathbb R^d$. 
Very recently, this result was extended to the $(p,q)$ setting with $p \geq q \geq d+1$ by Edwards and Sober{\'o}n, and subsequently to the case $p \geq q \geq 2$ by Gangopadhyay, Polyanskii, and the author of this paper.

In this paper, we obtain improved bounds for the $(p,q)$ problem in the case $q=2$ and $d=2$. 
More precisely, our main result asserts that for any integer $p \geq 2$, any set $P \subseteq \mathbb R^2$, and any finite family $\mathcal B$ of axis-parallel rectangles in $\mathbb R^2$ such that every rectangle contains a point of $P$, if among every $p$ rectangles there exist two whose intersection contains a point of $P$, then there exists a subset $S \subseteq P$ of size at most $O\!\bigl( (p \log \log p)^2 \bigr)$ such that every rectangle contains a point of $S$. 
Moreover, when $p=2$, the size of $S$ can be bounded by $4$.
   
\end{abstract}


\section{Introduction}
The celebrated theorem of Helly~\cite{helly1923mengen} concerns the intersection patterns of families of convex sets in \(\mathbb{R}^d\).
It states that for a finite family \(\mathcal F\) of convex sets in \(\mathbb{R}^d\), if every subfamily of at most \(d+1\) sets has nonempty intersection, then all members of \(\mathcal F\) intersect.
Over the years, numerous results related to Helly’s theorem have been obtained; see the survey~\cite{barany2022helly}.

Helly’s theorem implies the following Helly-type result for \emph{axis-parallel boxes} (or simply \emph{boxes}), which are defined as Cartesian products of \(d\) closed intervals in \(\mathbb{R}\).
Given a finite family \(\mathcal B\) of boxes in \(\mathbb{R}^d\), if every two members of \(\mathcal B\) intersect, then $\bigcap \mathcal B \coloneqq \bigcap_{B \in \mathcal B} B \neq \emptyset$.

In 2008, Halman~\cite{Hal} proved the following discrete Helly-type result for boxes.
For a set \(X\) and a family \(\mathcal F\) of sets, the \emph{trace} of \(\mathcal F\) on \(X\), denoted by \(\mathcal F|_X\), is defined as $\mathcal F|_X \coloneqq \{ F \cap X : F \in \mathcal F \} $.

\begin{theorem}[Halman’s theorem; Theorem~2.10 in~\cite{Hal}]
\label{Halman theorem}
Let \(d\) be a positive integer.
Let \(P\) be a finite set in \(\mathbb{R}^d\), and let \(\mathcal B\) be a finite family of boxes in \(\mathbb{R}^d\).
If for every subfamily \(\mathcal B' \subseteq \mathcal B\) of size at most \(2d\) the trace \(\mathcal B'|_P\) is intersecting, then \(\mathcal B|_P\) is also intersecting.
\end{theorem}

Very recently, Edwards and Sober{\'o}n~\cite{edwards2025extensions} proved a \((p,q)\)-theorem for discrete boxes with \(p \ge q \ge d+1\).
Subsequently, Gangopadhyay, Polyanskii, and the author of this paper improved this result to the case \(p \ge q \ge 2\) in~\cite{gangopadhyay2025new}. 

Given a family $\mathcal F$ and integers $p\geq q$, we say it has \textit{$(p,q)$-property}, if among every $p$ members of $\mathcal F$, there are $q$ members that intersect. The \textit{piercing} (or \textit{transversal}) number of $\mathcal F$, denoted by $\tau(\mathcal F)$, is the smallest size of a set $S$ that intersects every member of $\mathcal F$. The \textit{packing} (or \textit{matching}) number of $\mathcal F$, denoted by $\nu (\mathcal F)$, is the largest size of a subfamily $\mathcal F' \subseteq \mathcal F$ such that members of $\mathcal F'$ are pairwise disjoint.

\begin{theorem}[Theorem~5 in~\cite{gangopadhyay2025new}]\label{theorem: theorem 5 in GPR25}
Let \(p,q,d\) be positive integers with \(p \ge q \ge 2\).
There exists \(N \coloneqq N(p,q,d)\) such that for any finite set \(P \subseteq \mathbb{R}^d\) and any finite family \(\mathcal B\) of boxes in \(\mathbb{R}^d\), the following holds.
If the trace \(\mathcal B|_P\) has the \((p,q)\)-property and does not contain the empty set, then $\tau(\mathcal B|_P) \le N $.
\end{theorem}

However, since the proof of Theorem~\ref{theorem: theorem 5 in GPR25} relies on the Alon--Kleitman argument~\cite{alon1992piercing} together with additional applications of Ramsey theory \cite{ramsey1930problem}, the resulting upper bound on \(N\) is very large when compared with the bounds known for standard boxes in \cite{tomon2023lower}.

\begin{theorem}[Theorem~4.1 in~\cite{tomon2023lower}]
\label{theorem: piercing rectangles}
Let \(\mathcal B\) be a finite family of boxes in \(\mathbb{R}^d\), where $d \geq 2$ is an integer.
Then
\[
\tau(\mathcal B) =
O\!\left(
\nu(\mathcal B)\,
\bigl(\log \nu(\mathcal B)\bigr)^{d-2}
\bigl(\log \log \nu(\mathcal B)\bigr)^2
\right),
\]

\end{theorem}

In this paper, we establish comparatively small upper bounds for \(N(2,2,2)\) and \(N(p,2,2)\), where $N(p,q,d)$ is the minimal universal upper bound on the piercing number $\tau(\mathcal B|_P)$ for any set $P\subseteq \mathbb R^d$ and any finite family $\mathcal B$ of axis-parallel boxes in $\mathbb R^d$ such that $\mathcal B|_P$ has the $(p,q)$-property and contains no empty set.

\begin{theorem}\label{theorem: N(2,2,2)}
We have \(N(2,2,2) \le 4\).
\end{theorem}

\begin{theorem}\label{theorem: N(p,2,2)}
We have $ N(p,2,2)= \Omega (p \log p)$ and $  N(p,2,2) =  O\!\bigl( (p\log  \log p)^2 \bigr)$.
\end{theorem}

Unfortunately, our method does not seem to extend to higher dimensions, and obtaining better bounds in higher dimensions may require more complicated tools.

The paper is organized as follows.  
In Section~\ref{section: preliminary}, we introduce \(d\)-intervals and the polytopal KKMS theorem, which are used in the proofs.  
Section~\ref{section: proof of theorem N(2,2,2)} is devoted to the proof of Theorem~\ref{theorem: N(2,2,2)}.  
In Section~\ref{section: proof of theorem N(p,2,2)}, we prove Theorem~\ref{theorem: N(p,2,2)}.

\section{Preliminaries}\label{section: preliminary}
Let \(L_1,\dots,L_d\) be \(d\) pairwise disjoint parallel lines in the plane\footnote{One may replace pairwise disjoint lines by pairwise disjoint homeomorphic copies of the real line.}.  
A \emph{\(d\)-interval} is a set consisting of \(d\) convex components \(I^{1},\dots,I^{d}\) such that \(I^{i} \subseteq L_i\) for each \(i \in [d]\).

\begin{theorem}[Theorem~1.4 in~\cite{kaiser1997transversals}]
\label{theorem: discrete d-intervals}
Let \(\mathcal I\) be a finite family of \(d\)-intervals with \(d \ge 2\).
Then $\tau(\mathcal I) \le (d^2 - d)\, \nu(\mathcal I)$.
\end{theorem}

We will also use the following polytopal KKMS theorem due to Komiya~\cite{komiya1994simple}.

\begin{theorem}[Theorem~2 in~\cite{komiya1994simple}]
\label{theorem: Komiya's extension}
Let \(Q\) be a polytope, and suppose that for every face \(F\) of \(Q\) a point \(q(F) \in F\) is chosen.
Let \(B_F\) be an open subset of \(Q\) assigned to each face \(F\), such that every face \(F'\) of \(Q\) satisfies $F' \subseteq \bigcup_{F \subseteq F'} B_F $.
Then there exists a collection \(\mathcal F\) of faces of \(Q\) such that $q(Q) \in \conv \{ q(F) : F \in \mathcal F \}
\quad \text{and} \quad
\bigcap_{F \in \mathcal F} B_F \neq \emptyset $.
\end{theorem}

\section{Proof of Theorem~\ref{theorem: N(2,2,2)}}\label{section: proof of theorem N(2,2,2)}

The proof is organized as follows. First, we present a stronger auxiliary statement, whose special case \(p=2\) implies Theorem~\ref{theorem: N(2,2,2)}. We then reduce the problem from rectangles to \(4\)-intervals. Finally, adapting the argument used in the proof of Theorem~6.3 in~\cite{aharoni2017fractional} for \(d\)-intervals, we prove the auxiliary statement.
We include several explanatory remarks to clarify the underlying ideas of the topological argument.

We now present the following stronger auxiliary statement.

\begin{theorem}\label{theorem: N(p,2,2) with every two intersect}
Let $p \geq 2$ be an integer, and let \(P \subseteq \mathbb{R}^2\). 
Suppose \(\mathcal{B}\) is a finite family of rectangles in \(\mathbb{R}^2\) satisfying the following conditions:
\begin{enumerate}
    \item every rectangle contains at least one point of \(P\);
    \item every pair of rectangles in \(\mathcal B\) has nonempty intersection;
    \item among any \(p\) rectangles, there exist $2$ whose intersection contains a point of \(P\).
\end{enumerate}
Then there exists a subset \(S \subseteq P\) of size at most \(4(p-1)\) such that
\(B \cap S \neq \emptyset\) for every \(B \in \mathcal{B}\).
\end{theorem}

\begin{proof}[Proof of Theorem~\ref{theorem: N(p,2,2) with every two intersect}]
    Since \(\mathcal{B}\) is finite, we may assume that \(P\) is finite as well. 
As every two rectangles in \(\mathcal{B}\) intersect, a classical application of Helly’s theorem implies that all rectangles in \(\mathcal{B}\) have a common intersection. 
Choose the origin to be a point in \(\bigcap \mathcal{B}\), and assume that for every rectangle in \(\mathcal{B}\), its four vertices lie in the four quadrants\footnote{Here we adopt the standard convention for the four quadrants of the plane with some changes: the first quadrant consists of points with nonnegative $x$- and $y$-coordinates, and the remaining quadrants are labeled in counterclockwise order (second, third, and fourth quadrants, respectively). Hence, each quadrant is taken to be closed, meaning that the two boundary rays (the coordinate axes) are included in the corresponding quadrant.}, respectively. 

Suppose that
\[
\bigcap \mathcal{B} = [a,b] \times [c,d].
\]
If there exists a point \(p \in P\) such that \(p \in \bigcap \mathcal{B}\), then the statement follows immediately. 
Hence, we assume that no point of \(P\) lies in \(\bigcap \mathcal{B}\).

    We now begin the reduction from rectangles to \(4\)-intervals. 
Since both \(P\) and \(\mathcal{B}\) are finite, we may assume that \(P\) and all rectangles in \(\mathcal{B}\) are contained in a bounding box $[a',b'] \times [c',d']$ for some \(a',b',c',d' \in \mathbb{R}\) without changing the intersection graph of $\mathcal B|_P$.

For any two points \(p_1=(x_1,y_1)\) and \(p_2=(x_2,y_2)\) of \(P\) lying in the same quadrant, we define
\(p_1 \preceq p_2\) if \(|x_1|\le |x_2|\) and \(|y_1|\le |y_2|\).
This relation defines a partial order on the points of \(P\) within each quadrant. 
Moreover, if \(p_2 \in B\) for some rectangle \(B\) and \(p_1 \preceq p_2\), then necessarily \(p_1 \in B\).
Consequently, it suffices to consider points of \(P\) that are minimal with respect to this order in their quadrant.
That is, we may assume that for any two points \(p_1,p_2 \in P\) lying in the same quadrant, neither \(p_1 \preceq p_2\) nor \(p_2 \preceq p_1\) holds; such points are said to be \emph{incomparable}.

Fix a quadrant and order the points of \(P\) assigned to this quadrant by increasing absolute value of their \(x\)-coordinates.
Let these points be \(p_1,\dots,p_k\).
Choose an auxiliary point \(p_0\) on the boundary ray of the \(y\)-axis belonging to this quadrant, outside the bounding box \([a',b']\times[c',d']\), and choose an auxiliary point \(p_{k+1}\) on the boundary ray of the \(x\)-axis belonging to this quadrant, also outside the bounding box. These auxiliary points are chosen so that the polygonal chain connecting
\[
p_0,p_1,\dots,p_k,p_{k+1}
\]
is monotone in the \(x\)-coordinate and in the absolute value of the \(y\)-coordinate. Let \(L\) denote this polygonal chain. Since every rectangle in \(\mathcal B\) is contained in \([a',b']\times[c',d']\), no rectangle in \(\mathcal B\) contains \(p_0\) or \(p_{k+1}\).

\begin{observation}
For any rectangle \(B \in \mathcal{B}\), if two points \((x_1,y_1),(x_2,y_2) \in B \cap L\) with \(x_1 < x_2\), then every point \((x_3,y_3) \in L\) satisfying \(x_1 \le x_3 \le x_2\) also belongs to \(B\).
\end{observation}

\begin{figure}[!h]
    \centering
    
    \includegraphics[scale=0.5]{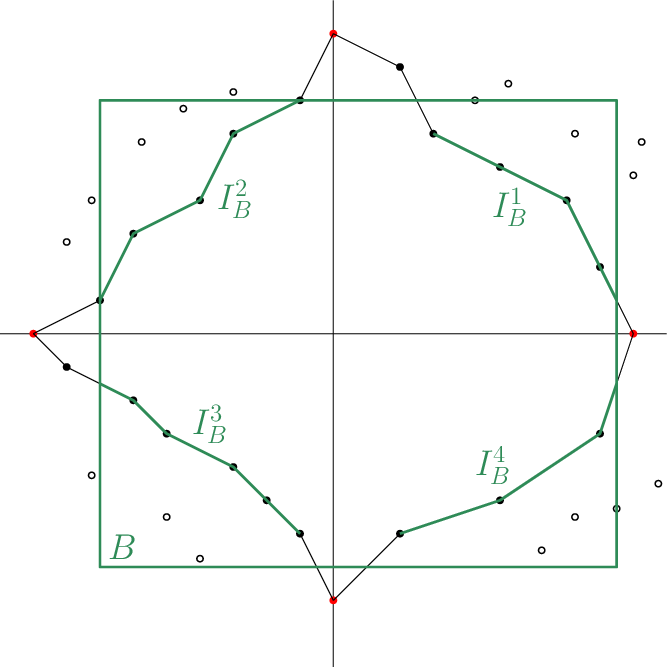}
    
    \caption{Illustration of reduction}
    \label{fi2}
\end{figure}

    Let \(L_1,\dots,L_4\) be the four polygonal chains corresponding to the four quadrants constructed above. By choosing the auxiliary endpoints sufficiently far from the bounding box and consistently on adjacent axes, we may arrange that
\[
\Pi_y(L_1)=\Pi_y(L_2),\quad 
\Pi_x(L_2)=\Pi_x(L_3),\quad 
\Pi_y(L_3)=\Pi_y(L_4),\quad 
\Pi_x(L_4)=\Pi_x(L_1),
\]
where \(\Pi_x\) and \(\Pi_y\) denote the projections onto the \(x\)- and \(y\)-axes, respectively. The auxiliary endpoints of these chains lie outside the bounding box and hence are contained in no rectangle of \(\mathcal B\).

For each rectangle \(B= [x_1,x_2] \times [y_1,y_2] \in \mathcal{B}\), suppose that \(B\) contains points
\[
p_{i_1^j}, \dots, p_{i_{k_j}^j}
\]
of \(P\) in the \(j\)-th quadrant, for \(j \in [4]\), where
\(i_1^j \le \dots \le i_{k_j}^j\). Let $q_1, q_4$ (respectively, $q_2, q_3$) be the intersection points of the line $\{(x,y) \in \mathbb{R}^2 : x = x_2\}$ (respectively, $\{(x,y) \in \mathbb{R}^2 : x = x_1\}$) with the polygonal chains $L_1, L_4$ (respectively, $L_2, L_3$). Let \(I_B\) be the union of the four polygonal chains \(I_B^j \subseteq L_j\) joining the points \(\{p_{i_1^j}, \dots, p_{i_{k_j}^j}, q_j\}\) in increasing order of their \(x\)-coordinates\footnote{If \(B\) contains no point of \(P\) in the \(j\)-th quadrant, we define \(I_B^j = \emptyset\).}.
Define $\mathcal{I} \coloneqq \{ I_B : B \in \mathcal{B} \}$. Clearly, among every $p$ members of $\mathcal I$, there are $2$ having a point in common.

\begin{claim}
$\tau(\mathcal{B}|_P) \le \tau(\mathcal{I}|_P) = \tau(\mathcal{I})$.
\end{claim}

\begin{proof}
Since \(I_B \subseteq B\) for every \(I_B \in \mathcal I\), any piercing set for \(\mathcal I|_P\) is also a piercing set for \(\mathcal B|_P\). Hence,
\[
\tau(\mathcal{B}|_P) \le \tau(\mathcal{I}|_P).
\]

We now prove \(\tau(\mathcal{I}|_P) = \tau(\mathcal{I})\).
The inequality \(\tau(\mathcal{I}|_P) \ge \tau(\mathcal{I})\) is immediate, since \(\mathcal{I}|_P\) is a restriction of \(\mathcal I\).
It therefore suffices to show that every piercing set of \(\mathcal I\) can be transformed into a piercing set of \(\mathcal I|_P\) of no larger size.

Let \(S\) be a piercing set of \(\mathcal I\).
For each point \(s \in S\), suppose \(s \in L_j\) for some \(j \in [4]\), and let
\[
\mathcal I_s \coloneqq \{ I \in \mathcal I : s \in I \}.
\]
Then \(s \in \bigcap_{I \in \mathcal I_s} I^j\), so this intersection is nonempty.
By the construction of \(\mathcal I\), at least one of the endpoints of the interval
\(\bigcap_{I \in \mathcal I_s} I^j\) belongs to \(P\). Indeed, if \(j \in \{1,4\}\), then the left endpoint belongs to \(P\); if \(j \in \{2,3\}\), then the right endpoint belongs to \(P\).
Hence, we may replace \(s\) by one of these endpoints, say \(s'\), such that
\[
s' \in \bigcap_{I \in \mathcal I_s} I \cap P.
\]

Repeating this replacement for every \(s \in S\) yields a piercing set
\(S' \subseteq P\) for \(\mathcal I|_P\) with \(|S'| \le |S|\).
Therefore,
\(\tau(\mathcal{I}|_P) \le \tau(\mathcal{I})\), completing the proof.
\end{proof}

Therefore, it suffices to show that \(\tau(\mathcal{I}) \le 4(p-1)\). Assume, for contradiction, that \(\tau(\mathcal{I}) > 4(p-1)\). 
We shall show that \(\nu(\mathcal{I}) \ge p\), contradicting the fact that among every $p$ members of \(\mathcal{I}\), there exist $2$ of them that intersect.

We apply Theorem~\ref{theorem: Komiya's extension} to the polytope
\(Q=\Delta_{p-1} \times \Delta_{p-1}\).
Each point \(q \in Q\) can be written as
\[
q = \bigl((q_1^1, \dots, q_{p}^1),(q_1^2, \dots, q_{p}^2)\bigr),
\]
where \(q_i^j \ge 0\) and \( \sum_{i=1}^p q_{i}^j = 1\) for \(j \in \{1,2\}\). For $q \in Q$, $i\in [p]$ and $j \in [2]$, let $r_q(i,j)= \sum_{k \leq i}q_k^j$.
We associate each \(q \in Q\) with $4(p-1)$ points \(l_{q,i}^n\) for $i \in[p-1]$ and $n\in[4]$ on the polygonal chains \(L_1,\dots,L_4\) as follows.

If $n \in \{1,4\}$ ($n \in \{2,3\}$, respectively), let \(l_{q,i}^n\) be the point on \(L_n\) whose \(x\)-coordinate is
\(|\Pi_x(L_n)|\, r_q(i,1)\) (\( - |\Pi_x(L_n)|\, r_q(i,2)\), respectively), where \(|\Pi_x(L_n)|\) denotes the length of the segment \(\Pi_x(L_n)\).

For each \(n\in[4]\), the points \(l_{q,1}^n,\dots,l_{q,p-1}^n\) split \(L_n\) into \(p\) open parts\footnote{Some of these parts may be empty when the corresponding coordinate of $q$ is zero.}, denoted by
\[
(l_{q,0}^n,l_{q,1}^n),\ (l_{q,1}^n,l_{q,2}^n),\ \dots,\ (l_{q,p-1}^n,l_{q,p}^n),
\]
where \(l_{q,0}^n\) and \(l_{q,p}^n\) denote the two auxiliary endpoints of \(L_n\). We call these the first through \(p\)-th parts, respectively.

\begin{figure}[!h]
    \centering
    
    \includegraphics[scale=0.5]{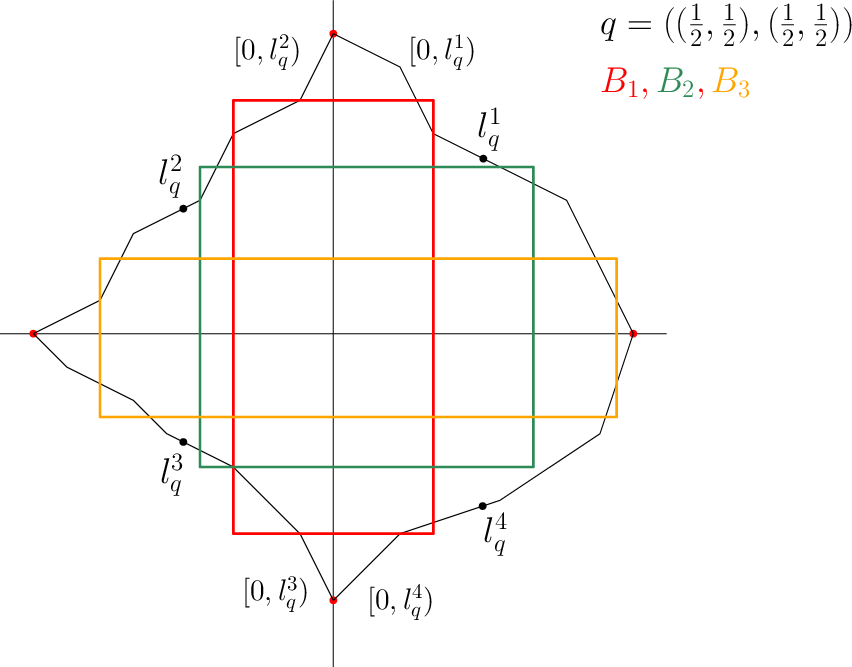}
    
    \caption{Illustration of the topological argument with $p=2$. In the figure, we consider the point 
$q=\bigl((\tfrac12,\tfrac12),(\tfrac12,\tfrac12)\bigr)$. Moreover, we have \(q \in B_{v_1}, B_{v_2}, B_{v_3}\), where $v_1 = ((1,0),(1,0)), 
v_2 = ((0,1),(1,0)), v_3 = ((0,1),(0,1))$,
due to the existence of corresponding sets \(I_{B_1}, I_{B_2}, I_{B_3} \in \mathcal I\), respectively.}
    \label{fi3}
\end{figure}

Let $V \coloneqq \{(i,j) : i \in [p],j \in [2]\}$,
and let \(G\) be the complete bipartite graph with vertex set \(V\) and bipartition
\(V^j = \{(i,j) : i \in [p]\}\).
Each vertex $v=\bigl((v_1^1,\dots, v_p^1),(v_1^2, \dots ,v_p^2)\bigr) \in Q$ satisfies \(v_{i_1}^1=1\) and \(v_{i_2}^2=1\) for some \(i_1,i_2 \in [p]\).
Thus, each vertex \(v \in Q\) can be bijectively associated with the edge $e_v = \{(i_1,1),(i_2,2)\}$ of \(G\).

\begin{remark*}
    Recall that each point $q \in Q$ is associated with $4(p-1)$ points $l^n_{q,i}$, where $i \in [p-1]$ and $n \in [4]$, on the four polygonal chains. 
For any $n \in [4]$, $p-1$ points $l^n_{q,i}$, where $i \in [p-1]$, divides $L_n$ into $p$ parts.
Informally, each vertex $v \in Q$, associated with the edge $e_v = \{(i_1,1),(i_2,2)\}$, is defined to encode a pattern indicating that there exists some $I \in \mathcal I$ such that \(I^1\) and \(I^4\) lie in the \(i_1\)-th parts of \(L_1\) and \(L_4\), respectively, while \(I^2\) and \(I^3\) lie in the \(i_2\)-th parts of \(L_2\) and \(L_3\), respectively.
\end{remark*}

To apply Theorem~\ref{theorem: Komiya's extension}, define \(q(v)=v\) for every vertex \(v \in Q\), and set
\[
q(Q) = \bigl((\tfrac1p, \dots, \tfrac1p),(\tfrac1p, \dots,\tfrac1p)\bigr).
\]
For any other face \(F \subseteq Q\), choose \(q(F) \in F\) arbitrarily, as these choices will not affect the argument.

For a vertex \(v\in Q\) associated with the edge \(\{(i_1,1),(i_2,2)\}\), define \(B_v\) to be the set of all points \(q \in Q\) for which there exists
\(I=\bigcup_{j\in[4]} I^j \in \mathcal{I}\) satisfying
\[
I^1 \subseteq (l_{q,i_1-1}^1,l_{q,i_1}^1),\quad
I^4 \subseteq (l_{q,i_1-1}^4,l_{q,i_1}^4),
\]
and
\[
I^2 \subseteq (l_{q,i_2-1}^2,l_{q,i_2}^2),\quad
I^3 \subseteq (l_{q,i_2-1}^3,l_{q,i_2}^3).
\]

We say that such an \(I\) \emph{escapes} through the pattern defined by \(q\) and \(v\). More specifically, \(I^1\) and \(I^4\) escape through the \(i_1\)-th parts of \(L_1\) and \(L_4\), respectively, while \(I^2\) and \(I^3\) escape through the \(i_2\)-th parts of \(L_2\) and \(L_3\), respectively. Clearly, each \(B_v\) is open, since all \(I^j\) are closed. For all other faces \(F \subseteq Q\), set \(B_F=\emptyset\).

\begin{remark*}
    Informally, a point $q$ determines the positions of “holes” on the polygonal chains $L_1,\dots,L_4$. For any $I \in \mathcal I$ satisfying that $I \cap \{l_{q,i}^n : i\in [p-1],n\in[4]\} =\emptyset$, there are in total $p^4$ possible ways to escape from these holes. The definition of \(B_v\) selects $p^2$ of these possibilities.
Every vertex $v$ specifies a pattern indicating that there exists some $I \in \mathcal I$ such that $I^1,\dots,I^4$ escape from the corresponding holes. 
Accordingly, the set $B_v$ consists of all points $q \in Q$ for which there exists some $I \in \mathcal I$ that escapes the holes defined by $q$ according to the pattern specified by $v$.
\end{remark*}

We now verify the conditions of Theorem~\ref{theorem: Komiya's extension}.
Since \(\tau(\mathcal{I}) > 4(p-1)\), for any $q \in Q$, there exists $I \in \mathcal I$ such that $I \cap \{l_{q,i}^n : i\in [p-1], n \in[4]\} =\emptyset$. By the definitions of \(\mathcal{I}\) and $B_v$, the set $I$ must escape from the holes defined by $q$ and some vertex $v \in Q$. 

Indeed, suppose that $I$ does not escape from the hole in any of the $p^2$ ways specified by the vertices of $Q$. Then either $I^1, I^4 \neq \emptyset$ with $I^1,I^4$ escape from the holes with different indices, or $I^2, I^3 \neq \emptyset$ with $I^2,I^3$ escape from the holes with different indices. Without loss of generality, assume that $I^1, I^4 \neq \emptyset$ with $I^1,I^4$ escape from the holes with different indices, as the remaining cases are analogous. However, this configuration contradicts the definition of $\mathcal I$, which requires that the right endpoints of $I^1$ and $I^4$ have the same $x$-coordinates, whenever both intervals are nonempty.

It follows that for any $q \in Q$, there exists \(v \in Q\) such that \(q \in B_v\).
Hence,
\[
Q \subseteq \bigcup_{v \in \mathrm{vert}(Q)} B_v.
\]
Every face of \(Q\) is of the form
\[
F(M^1,M^2) \coloneqq \{q : q_m^j=0 \text{ for all } m \notin M^j,\ j\in[2]\},
\]
where \(M^j \subseteq [p]\).
Fix a nonempty face \(F(M^1,M^2)\), and let
\[
q=\big((q_1^1,\dots,q_p^1),(q_1^2,\dots,q_p^2)\big)\in F(M^1,M^2).
\]
Choose a vertex \(v=\bigl((v_1^1,\dots,v_p^1),(v_1^2,\dots,v_p^2)\bigr)\) such that \(q\in B_v\), and let \(i_1,i_2\in[p]\) be defined by \(v^1_{i_1}=1\) and \(v^2_{i_2}=1\).

Suppose that \(q_{i_1}^1 > 0\) and \(q_{i_2}^2 > 0\). Then \(i_1 \in M^1\) and \(i_2 \in M^2\). Hence, by the definition of \(F(M^1,M^2)\), we have \(v \in F(M^1,M^2)\).

Next, suppose that \(q_{i_1}^1 = 0\) and \(q_{i_2}^2 > 0\). Then the \(i_1\)-th parts of \(L_1\) and \(L_4\) are empty. Since \(q \in B_v\), by the definition of \(B_v\), there exists an interval \(I \in \mathcal I\) such that \(I^1\) and \(I^4\) escape through the \(i_1\)-th parts of \(L_1\) and \(L_4\). This implies that \(I^1\) and \(I^4\) are empty. Choose $i_1'$ such that $q^1_{i_1'}>0$. Consequently, \(I\) can also escape through the holes specified by the vertex
\[
u = \bigl((u_1^1,\dots,u_p^1),(u_1^2,\dots,u_p^2)\bigr),
\]
where \(u^1_{i_1'} = 1\) and \(u^2_{i_2} = 1\). Hence, \(q \in B_u\). Since \(q^1_{i_1'}  > 0\) and \(q^2_{i_2} > 0\), we have \(i_1' \in M^1\) and \(i_2 \in M^2\). Therefore, by the definition of \(F(M^1,M^2)\), we obtain \(u \in F(M^1,M^2)\).

The case \(q_{i_1}^1 > 0\) and \(q_{i_2}^2 = 0\) is analogous: one can similarly find a vertex \(u\) such that \(q \in B_u\) and \(u \in F(M^1,M^2)\). Finally, suppose that \(q_{i_1}^1 = 0\) and \(q_{i_2}^2 = 0\). Then all of \(I^1,\dots,I^4\) are empty, contradicting the fact that at least one of them is nonempty.

Thus, we conclude that for every \(F = F(M^1,M^2)\) and every \(q \in F\), there exists a vertex \(v\) of \(F\) such that \(q \in B_v\). It follows that
\[
F(M^1,M^2) \subseteq \bigcup \{ B_v : v \in F(M^1,M^2) \},
\]
and the hypotheses of Theorem~\ref{theorem: Komiya's extension} are satisfied.

Therefore, by Theorem~\ref{theorem: Komiya's extension} there exists a set \(T\) of vertices such that
\(q(Q) \in \operatorname{conv}\{q(v_i): v_i \in T\}\) and
\(\bigcap_{v_i \in T} B_{v_i} \neq \emptyset\).
Let \(E=\{e_{v_i} : v_i \in T\}\), where each $e_{v_i}$ is defined above for the complete bipartite graph $G$ on $V$, and consider the bipartite graph \(D=(V,E) \).
Since $q(v)=v$ and 
\[
q(Q)=\bigl((\tfrac1p,\dots,\tfrac1p),(\tfrac1p,\dots,\tfrac1p)\bigr) \in \operatorname{conv}\{q(v_i): v_i \in T\},
\]
there exist coefficients \(t_{v_i} \in [0,1]\) with $\sum_{v_i\in T} t_{v_i}=1$, for \(v_i \in T\), such that
\[
\sum_{v_i \in T} t_{v_i} v_i = \bigl((\tfrac1p,\dots,\tfrac1p),(\tfrac1p,\dots,\tfrac1p)\bigr).
\]
By the definition of \(e_{v_i}\) and the coordinates of the vertices \(v_i\), the function \(f : E \to \mathbb{R}_{\geq 0}\) defined by \(f(e_{v_i}) = pt_{v_i}\) satisfies
\[
\sum_{e \ni u} f(e) = 1 \quad \text{for every } u \in V.
\]
By double counting, we obtain
\[
\sum_{e \in E} f(e)
= \tfrac12 \sum_{e \in E} \sum_{u \in e} f(e)
= \tfrac12 \sum_{u \in V} 1
= p,
\]
and hence \(\nu^\ast(D) \ge p\).

Since \(D\) is bipartite, Kőnig’s theorem yields
\[
\nu(D)=\tau(D)\ge \tau^\ast(D)=\nu^\ast(D)\ge p.
\]
Let \(M\) be a matching in \(D\) with \(|M|\ge p\).
Choose \(q_0 \in \bigcap_{v \in T} B_v\).
By the definition of the sets \(B_v\), for each edge \(e_v \in M\) there exists
\(I_{e_v} \in \mathcal{I}\) that escapes from the holes defined by \(q_0\) and \(v\).
Note that these sets \(I_{e_v}\) are pairwise disjoint, which implies \(\nu(\mathcal{I}) \ge p\).
This contradiction completes the proof.

\end{proof}

\section{Proof of Theorem~\ref{theorem: N(p,2,2)}}\label{section: proof of theorem N(p,2,2)}

\subsection{The upper bound}
Since among every \(p\) rectangles, there exist two whose intersection contains a point of \(P\), we have
\(\nu(\mathcal B) \leq p-1\).
Hence, by Theorem~\ref{theorem: piercing rectangles}, there exists a piercing set for \(\mathcal B\) of size at most
\(O\!\Big((p-1) \big(\log\log(p-1)\big)^2 \Big)\).
This implies that \(\mathcal B\) can be partitioned into at most
\(O\!\Big((p-1) \big(\log\log(p-1)\big)^2 \Big)\) subfamilies, each of which has the property that every two rectangles intersect.

By Theorem~\ref{theorem: N(p,2,2) with every two intersect}, for each such subfamily there exists a subset
\(S \subseteq P\) of size at most \(O(p-1)\) that intersects every rectangle in the subfamily.
Combining these piercing sets over all subfamilies, we obtain
\[
\tau(\mathcal B|_P) = O\!\big( (p\log\log p)^2\big),
\]
which completes the proof.

\subsection{The lower bound}
In this subsection, we prove $N(p,2,2) = \Omega(p \log p)$. The construction is obtained from the result due to Pach and Tardos in~\cite{pach2010coloring}.

We first recall the part of their construction that we need. For integers
\(c_1 \ge 2\) and \(c_2 \ge 1\)\footnote{In~\cite{pach2010coloring}, they use parameters $c$ and $d$.}, Pach and Tardos construct a family $\mathcal R(c_1,c_2)$
of open axis-parallel rectangles with $|\mathcal R(c_1,c_2)|=(c_2+1)c_1^{c_2-1}$, satisfying the following lemma. 

Let \(H^\ast(c_1,c_2)\) be a hypergraph on vertex set
\(\mathcal R(c_1,c_2)\), where a nonempty set $\mathcal S\subseteq \mathcal R(c_1,c_2)$
is a hyperedge if and only if there exists a point of the plane covered by
exactly the rectangles in \(\mathcal S\), and by no other rectangle of
\(\mathcal R(c_1,c_2)\).

\begin{lemma}[Theorem~3 in \cite{pach2010coloring}]\label{lemma: lemma due to Pach and Tardos}
    Let $c_2 \geq 1$, $2\leq r<c_1$, and let $H^\ast=H^\ast(c_1,c_2)$ be the hypergraph defined as above. If a subset $I \subseteq \mathcal R(c_1,c_2)$ contains no hyperedge of $H^\ast$ of size $r$, then we have \[ |I| \leq \frac{c_1^{c_2-1}}{\frac{1}{r-1}-\frac{1}{c_1-1}}. \]
\end{lemma}

Let $G^\ast(c_1,c_2)=H^\ast_2(c_1,c_2)$ be the graph consisting of all two-element hyperedges of \(H^\ast(c_1,c_2)\). Thus
two rectangles \(R,R'\in\mathcal R(c_1,c_2)\) are adjacent in \(G^\ast(c_1,c_2)\) if and
only if there exists a point covered by exactly \(R\) and \(R'\), and by no other
rectangle of \(\mathcal R(c_1,c_2)\).

\begin{observation}\label{observation: independent number of}
    For every integer \(c_1\ge 3\) and every \(c_2\ge 1\), the independence number of $G^\ast(c_1,c_2)$ satisfies 
\[
\alpha(G^\ast(c_1,c_2))
\le
\frac{c_1-1}{c_1-2}\,c_1^{c_2-1}.
\]
In particular,
\[
\alpha(G^\ast(3,c_2))\le 2\cdot 3^{c_2-1}.
\]
\end{observation}

\begin{proof}
Let \(I\subseteq \mathcal R(c_1,c_2)\) be an independent set in \(G^\ast(c_1,c_2)\).
Since the edges of \(G^\ast(c_1,c_2)\) are precisely the two-element hyperedges of
\(H^\ast(c_1,c_2)\), the set \(I\) contains no hyperedge of \(H^\ast(c_1,c_2)\) of size
\(2\). Then the observation follows from Lemma~\ref{lemma: lemma due to Pach and Tardos}.
\end{proof}

We now prove the following theorem, which implies the desired lower bound.

\begin{theorem}
\label{theorem: lower bound}
For every integer $p\geq 3$, there exist a finite family $\mathcal{R}$ of axis-parallel rectangles in $\mathbb{R}^2$ and a set $P \subseteq \mathbb{R}^2$ such that the trace $\mathcal{R}|_P$ has the $(p,2)$-property, $\emptyset \notin \mathcal{R}|_P$, and
\[
\tau(\mathcal{R}|_P) > \frac{p-1}{12} \log_3 \frac{p-1}{6}.
\]
\end{theorem}

\begin{proof}
We use the Pach--Tardos construction with \(c_1=3\). Put
\[
\mathcal R_{c_2}=\mathcal R(3,c_2),
\qquad
G_{c_2}=G^\ast(3,c_2).
\]
Then $|\mathcal R_{c_2}|=(c_2+1)3^{c_2-1}$
and by Observation~\ref{observation: independent number of}, $\alpha(G_{c_2})\le 2\cdot 3^{c_2-1}$.

Let \(\mathcal R'_{c_2}\subseteq\mathcal R_{c_2}\) be the set of non-isolated vertices
of \(G_{c_2}\), and let $G'_{c_2}=G_{c_2}[\mathcal R'_{c_2}]$ be the induced subgraph on \(\mathcal R'_{c_2}\). Since the isolated vertices of
\(G_{c_2}\) form an independent set, their number is at most $\alpha(G_{c_2})\le 2\cdot 3^{c_2-1}$.

Therefore,
\[
|\mathcal R'_{c_2}|
\ge
|\mathcal R_{c_2}|-2\cdot 3^{c_2-1}
=
(c_2+1)3^{c_2-1}-2\cdot 3^{c_2-1}
=
(c_2-1)3^{c_2-1}.
\]

We now construct a finite point set \(P_{c_2}\). For every edge
\(RR'\in E(G'_{c_2})\), choose one point $x_{RR'}$ that is covered by exactly the two rectangles \(R\) and \(R'\), and by no other rectangle of \(\mathcal R_{c_2}\). Such a point exists by the definition of
\(G^\ast(3,c_2)\). Define $ P_{c_2}=\{x_{RR'}:RR'\in E(G'_{c_2})\}$.

For every \(R\in\mathcal R'_{c_2}\), the trace \(R|_{P_{c_2}}\) is nonempty because
\(R\) is not isolated in \(G_{c_2}\). Moreover, by the definition of $P_{c_2}$, for any
\(R,R'\in\mathcal R'_{c_2}\), we have $R\cap R'\cap P_{c_2}\neq\emptyset$
if and only if $RR'\in E(G'_{c_2})$.

It follows that a subfamily of the trace family \(\mathcal R'_{c_2}|_{P_{c_2}}\) is
pairwise disjoint if and only if the corresponding vertices form an independent
set in \(G'_{c_2}\). Hence
\[
\nu(\mathcal R'_{c_2}|_{P_{c_2}})
=
\alpha(G'_{c_2})
\le
\alpha(G_{c_2})
\le
2\cdot 3^{c_2-1}.
\]

On the other hand, every point of \(P_{c_2}\) is covered by exactly two rectangles of
\(\mathcal R'_{c_2}\). Therefore one point of \(P_{c_2}\) can pierce at most two traces.
Consequently,
\[
\tau(\mathcal R'_{c_2}|_{P_{c_2}})
\ge
\frac{|\mathcal R'_{c_2}|}{2}
\ge
\frac{(c_2-1)3^{c_2-1}}{2}.
\]

Now for any integer $p\geq 3$, there exists $c_2 \geq 1$ such that $2\cdot 3^{c_2-1}+1 \leq p <2\cdot 3^{c_2}+1$. Consider the family $\mathcal R'_{c_2}|_{P_{c_2}}$ constructed above. Since $\nu(\mathcal R'_{c_2}|_{P_{c_2}}) \leq 2\cdot 3^{c_2-1}$, the family $\mathcal R'_{c_2}|_{P_{c_2}}$ has the $(p,2)$-property. Moreover, since $p <2\cdot 3^{c_2}+1$, we have \[\tau(\mathcal R'_{c_2}|_{P_{c_2}}) \geq  \frac{(c_2-1)3^{c_2-1}}{2} > \frac{p-1}{12} \log_3\frac{p-1}{6}, \] which completes the proof.
\end{proof}

\section*{Acknowledgements}
I would like to thank Minki Kim for drawing my attention to the paper~\cite{tomon2023lower}, as well as Alexander Polyanskii, Mikhail Bludov and Rahul Gangopadhyay for fruitful discussions.

\bibliographystyle{alpha}
\bibliography{references}

\end{document}